\documentclass[12pt, a4paper]{amsart}

\usepackage[english]{babel}
\usepackage{amsmath}
\usepackage{amssymb}
\usepackage{bm}
\usepackage{amscd} 
\usepackage{microtype} 
\usepackage{verbatim} 
\usepackage{color}  
\usepackage{tikz-cd}\usetikzlibrary{babel} 
\usepackage{enumitem}
\usepackage{mathtools} 
\usepackage{cite}
\usepackage{hyperref} 
\usepackage[initials,msc-links,backrefs]{amsrefs}
\usepackage[all]{xy}

\usepackage[OT2, T1]{fontenc}

\title[Profinitely solitary Chevalley groups]{Galois cohomology and profinitely solitary Chevalley groups}

 \author[H. Kammeyer]{Holger Kammeyer}
 \author[R. Spitler]{Ryan Spitler}
 
 \address{Mathematical Institute, University of D{\"u}sseldorf, Germany}
 \email{holger.kammeyer@hhu.de}

 \address{Department of Mathematics, Rice University, Texas, USA}
 \email{rfs8@rice.edu}
 

\theoremstyle{plain}
\newtheorem{theorem}{Theorem}

\newtheorem{proposition}[theorem]{Proposition}

\theoremstyle{definition}

\newtheorem*{definition*}{Definition}

\newtheorem*{observation*}{Observation}

\providecommand{\ignore}[1]{}

\providecommand{\R}{\mathbb{R}}
\providecommand{\Q}{\mathbb{Q}}
\providecommand{\Z}{\mathbb{Z}}

\providecommand{\C}{\mathbb{C}}

\newcommand*{\arXiv}[1]{ \href{http://www.arxiv.org/abs/#1}{arXiv:\textbf{#1}}}

\begin{document}

\begin{abstract}
For every number field and every Cartan Killing type, there is an associated split simple algebraic group. We examine whether the corresponding arithmetic subgroups are profinitely solitary so that the commensurability class of the profinite completion determines the commensurability class of the group among finitely generated residually finite groups.  Assuming Grothendieck rigidity, we essentially solve the problem by Galois cohomological means.
\end{abstract}

\maketitle

\section{Introduction}

A finitely generated residually finite group \(\Gamma\) is called \emph{(absolutely) profinitely solitary} if every other finitely generated residually finite group \(\Delta\) with \(\widehat{\Delta} \approx \widehat{\Gamma}\) satisfies \(\Delta \approx \Gamma\).  Here ``\(\approx\)'' means being abstractly commensurable and \(\widehat{\Gamma}\) is the profinite completion of \(\Gamma\).  The purpose of this article is to examine which arithmetic subgroups of split simple algebraic groups (sometimes known as \emph{Chevalley groups}) are profinitely solitary.  To this end, let \(k\) be a number field and denote the ring of finite adeles of \(k\) by \(\mathbb{A}^f_k\).  We say \(k\) is \emph{locally determined} if every number field \(l\) with \(\mathbb{A}^f_l \cong \mathbb{A}^f_k\) satisfies \(l \cong k\).

\begin{theorem} \label{thm:main-theorem}
Let \(k\) be a locally determined number field and let \(\mathbf{G}\) be a simply-connected absolutely almost simple split linear \(k\)-group such that 
      \begin{enumerate}[label=(\roman*)]
      \item \label{item:totima} \(k\) is totally imaginary and \(\mathbf{G}\) is not of type \(A_1\) if \(k\) is imaginary quadratic, or
      \item \label{item:a2n+1} \(k\) has precisely one real place, and either \(\mathbf{G}\) has type \(A_{2n+1}\) with \(n \ge 0\) (\(n \ge 1\) for \(k = \Q\)) or \(\mathbf{G}\) has type \(C_n\) with \(n \ge 2\), or
      \item \label{item:a2n} \(k\) is arbitrary and \(\mathbf{G}\) has type \(A_{2n}\) with \(n \ge 1\), or
      \item \label{item:qforms} \(k = \Q\) and \(\mathbf{G}\) has type \(B_3\), \(B_4\), \(D_4\), \(D_5\), or \(G_2\).
      \item \label{item:a1} \(\mathbf{G}\) has type \(A_1\) and \(k\) has signature \((2,0)\), \((3,0)\), or \((2,1)\).
  \end{enumerate}
Then either all arithmetic subgroups of \(\mathbf{G}\) are profinitely solitary, or there exists an arithmetic subgroup of \(\mathbf{G}\) that has a proper Grothendieck subgroup.
\end{theorem}

Recall that \(\Delta \le \Gamma\) is called a \emph{Grothendieck subgroup} if the inclusion induces an isomorphism \(\widehat{\Delta} \cong \widehat{\Gamma}\).  It is a more than 30 year old open problem whether higher rank arithmetic groups are \emph{Grothendieck rigid}, meaning they do not have proper Grothendieck subgroups~\cite{Platonov-Rapinchuk:algebraic-groups}*{p.\,434}.

To put the theorem into context, let us recall that replacing commensurability ``\(\approx\)'' with isomorphism ``\(\cong\)'' in the definition of profinite solitude yields the more established notion of \emph{profinite rigidity}.  The most prominent open cases of profinite rigidity are probably the noncommutative free groups \(F_n\) and the groups \(\operatorname{SL}_n(\mathbb{Z})\).  In this direction, and building on his joint work~\cite{Bridson-et-al:absolute} with Bridson--McReynolds--Reid, the second author proved in~\cite{Spitler:profinite} that for \(n \ge 3\), either \(\operatorname{SL}_n(\mathbb{Z})\) is profinitely rigid or not Grothendieck rigid.

Roughly, this is done by establishing that if a finitely generated residually finite group \(\Delta\) satisfies \(\widehat{\Delta} \cong \widehat{\Gamma}\) for an arithmetic group \(\Gamma\) of a higher rank simple \(k\)-group \(\mathbf{G}\), then \(\Delta\) embeds in an arithmetic subgroup \(\Lambda\) of an \(l\)-group \(\mathbf{H}\) such that \(\mathbb{A}^f_k \cong \mathbb{A}^f_l\), \(\mathbf{G}(\mathbb{A}^f_k) \cong \mathbf{H}(\mathbb{A}^f_l)\), and \(\widehat{\Gamma} \cong \widehat{\Lambda}\).  If \(\Gamma = \operatorname{SL}_n (\mathbb{Z})\) in \(\mathbf{G} = \mathbf{SL_n}\), one can conclude that \(k= \Q = l\) and invoke the classification of central simple algebras to determine that \(\mathbf{H} = \mathbf{SL_n}\) and \(\Lambda = \operatorname{SL}_n (\mathbb{Z})\).

\medskip
To conclude similarly for a general \(k\)-split group \(\mathbf{G}\), we would firstly need to know that \(k\) is locally determined to obtain an isomorphism \(\sigma \colon l \rightarrow k\).  We then have \(\mathbf{G}(\mathbb{A}^f_k)  \cong {}^\sigma \mathbf{H}(\mathbb{A}^f_k)\) which shows that \({}^\sigma \mathbf{H}\) splits at all finite places of \(k\).  From this we would secondly need to conclude that \({}^\sigma\mathbf{H}\) splits over \(k\).  This would give \({}^\sigma \mathbf{H} \cong_k \mathbf{G}\) by uniqueness of the \(k\)-split form, hence \(\Gamma \approx \Lambda\).  This second task can be described as verifying a ``Hasse local-global principle for being \(k\)-split'' with infinite places omitted.  In Galois cohomological terms, the task is to show that the diagonal map
  \[ g \colon H^1(k, \operatorname{Aut} \mathbf{G}) \longrightarrow \prod_{v \nmid \infty} H^1(k_v, \operatorname{Aut} \mathbf{G}) \]
  has trivial kernel.  When this occurs, we shall say the \(k\)-group \(\mathbf{G}\) satisfies the \emph{finite splitting principle}.  Identifying the cases in which this principle holds is the main technical result of this article.

\begin{theorem} \label{thm:finite-splitting-principle}
  Let \(\mathbf{G}\) be a simply-connected absolutely almost simple split linear \(k\)-group.  Then \(\mathbf{G}\) satisfies the finite splitting principle if and only if
    \begin{enumerate}[label=(\roman*)]
  \item \(k\) is totally imaginary and \(\mathbf{G}\) is arbitrary, or
  \item \(k\) has precisely one real place and \(\mathbf{G}\) has type \(A_{2n+1}\) or \(C_n\), or
  \item \(k\) is arbitrary and \(\mathbf{G}\) has type \(A_{2n}\). 
  \end{enumerate}
\end{theorem}

From this theorem, the first three cases of Theorem~\ref{thm:main-theorem} follow easily.  The small adaptations for type \(A_1\) in \ref{item:totima} and \ref{item:a2n+1} are in place to ensure \(\mathbf{G}\) has the \emph{congruence subgroup property} (CSP) which is an important assumption in~\cite{Spitler:profinite}.  CSP is also one of the reasons we are studying the concept of profinite solitude instead of rigidity because for a totally imaginary number field~\(k\), CSP does not hold in the strict sense.

Though Theorem~\ref{thm:finite-splitting-principle} is an ``if and only if'' statement, the first three cases of Theorem~\ref{thm:main-theorem} do not quite make up the entire list of profinitely solitary higher rank split arithmetic groups (assuming Grothendieck rigidity.)  For if \(k\) has few real places and \(\ker g\) is nontrivial, it might occur that all nontrivial elements in \(\ker g\) correspond to arithmetic groups which are low rank lattices.  This situation can be exploited in the last two cases of Theorem~\ref{thm:main-theorem} to still give the main result.   On the other hand, there remain some questions about type \(A_1\) for other signatures of \(k\) and type \(F_4\) for \(k = \Q\).  These questions are exclusively related to the incomplete status of Serre's conjecture on CSP as we address below.  Apart from these issues, Theorem~\ref{thm:main-theorem} gives the complete list of profinitely solitary Chevalley groups assuming Grothendieck rigidity:

  \begin{theorem} \label{thm:converse}
    Let \(\mathbf{G}\) be a simply-connected absolutely almost simple split linear \(k\)-group.  Assume that \(k\) is not locally determined or that \(k\) is locally determined and none of the cases \ref{item:totima}-\ref{item:a1} of Theorem~\ref{thm:main-theorem} applies.  Assume moreover \(\mathbf{G}\) is not of type \(A_1\), and not of type \(F_4\) if \(k=\Q\).  Then no arithmetic subgroup of \(\mathbf{G}\) is profinitely solitary.
  \end{theorem}

  For completeness, let us say what we expect for the cases not yet covered by either Theorem~\ref{thm:main-theorem} or~\ref{thm:converse}.  According to a well known conjecture by Serre, all higher rank lattices should have CSP.  Anisotropic \(A_1\)-forms constitute a notorious open case of this conjecture.  But if the conjecture does hold in this case, then all type \(A_1\) Chevalley groups over number fields of signature different from \((2,0), (3,0), (2,1)\) are not profinitely solitary because we find anisotropic higher rank forms in \(\ker g\).

  That leaves type \(F_4\) over \(\Q\) as the only case where the outcome is up for debate.  In that case, \(\ker g\) contains a lattice in the rank one Lie group \(F_{4(-20)}\).  Serre had originally conjectured that such lattices should not have CSP.  But it was meanwhile discovered that those lattices show ``higher rank behavior'' in other respects: property \((T)\), superrigidity, and arithmeticity.  One might take this as evidence that they should have CSP, too~\cite{Lubotzky:non-arithmetic}*{Section~4}.  If that is true, then Chevalley groups of type \(F_4\) over \(\Q\) are not profinitely solitary.  If on the other hand, lattices in \(F_{4(-20)}\) and their Zariski dense subgroups do not have CSP, then Chevalley groups of type \(F_4\) over \(\Q\) fall under the conclusion of Theorem~\ref{thm:main-theorem}.

  To make things explicit, we list some examples of solitary and non-solitary Chevalley groups, illustrating that the question of profinite solitude is more intricate than anticipated.

\bigskip
\noindent \textbf{Examples.} The following Chevalley groups are either profinitely solitary or not Grothendieck rigid:
  \begin{align*}
    A_n \colon & \operatorname{SL}_2(\Z[\sqrt{2}]), \ \operatorname{SL}_n(\mathbb{Z}) \text{ for } n \ge 3, \ \operatorname{SL}_3(\mathbb{Z}[\sqrt{7}]), \ \operatorname{SL}_4(\mathbb{Z}[\sqrt[3]{2}]), \\
    B_n \colon & \operatorname{Spin}(3,4)(\Z), \ \operatorname{Spin}(4,5)(\Z), \\
    C_n \colon & \operatorname{Sp}_n(\mathbb{Z}) \text{ and } \operatorname{Sp}_n(\mathbb{Z}[\sqrt[3]{2}]) \text{ for } n \ge 2, \\
    D_n \colon & \operatorname{Spin}(4,4)(\Z), \ \operatorname{Spin}(5,5)(\Z). \\
    \textup{Exceptional} \colon &  G_2(\Z). \\
    \intertext{In contrast, the following Chevalley groups are not profinitely solitary:}
    A_n \colon & \operatorname{SL}_3(\Z[\sqrt[8]{7}]), \ \operatorname{SL}_4(\Z[\sqrt{2}]), \\
    B_n \colon & \operatorname{Spin}(5 + n,6 + n)(\mathbb{Z}) \text{ for } n \ge 0, \\
    C_n \colon & \operatorname{Sp}_n(\Z[\sqrt{2}]) \text{ and } \operatorname{Sp}_n(\Z[\sqrt[4]{2}]) \text{ for } n \ge 2, \\
    D_n \colon & \operatorname{Spin}(6 + n,6 + n)(\mathbb{Z}) \text{ for } n \ge 0. \\
    \textup{Exceptional} \colon & E_6(\Z), \ E_7(\mathbb{Z}), \ E_8(\mathbb{Z}), \ F_4(\Z[\sqrt{2}]), \ G_2(\Z[\sqrt{2}]).
  \end{align*} 

\medskip
We point out that in~\cite{Kammeyer:absolute} and with S.\,Kionke in~\citelist{\cite{Kammeyer-Kionke:adelic-superrigidity} \cite{Kammeyer-Kionke:rigidity}}, the first author previously examined profinite solitude for arithmetic lattices in \emph{simple} Lie groups.  This entails strong restrictions on \(\mathbf{G}\) and \(k\), as \(\mathbf{G}\) must be anisotropic at all but one infinite place.  Hence these lattices are \emph{cocompact} unless \(k = \Q\).  Chevalley groups, in contrast, now provide many examples and counterexamples of the opposite nature: they are \emph{noncocompact} irreducible lattices in \emph{semisimple} Lie groups and these Lie groups are not simple unless \(k = \Q\) or \(k = \Q(\sqrt{-d})\).  Let us also remark that in \cite{Kammeyer:absolute}*{Theorem~3}, it was shown that some real form of type \(E_7\) has a non-profinitely solitary non-cocompact lattice.  But the precise real form was not identified, as the argument used the pigeonhole principle.  The proof of Theorem~\ref{thm:converse} now reveal that such a lattice exists in \(E_{7(7)}\) and \(E_{7(-25)}\).

Finally, we want to stress that while we saw that many Chevalley groups are profinitely solitary, this by no means implies that one should expect that all groups in the commensurability class of a Chevalley group are profinitely rigid.  In fact, they typically are not.  In the interesting recent preprint~\cite{Weiss:non-rigidity}, A.\,Y.\,Weiss Behar has compiled various methods to construct counterexamples to profinite rigidity within a commensurability class that nicely complement this work.

The outline of this article is as follows. Section~\ref{section:galois} collects necessary background on the Galois cohomology of simple algebraic groups.  Sections~\ref{section:finite-splitting-principle}, \ref{section:main-theorem}, and \ref{section:converse} prove Theorems~\ref{thm:finite-splitting-principle}, \ref{thm:main-theorem} and~\ref{thm:converse}, respectively.  Section~\ref{section:converse} also contains some explanations about the above list of examples.

\medskip
The first author acknowledges financial support from the German Research Foundation within the Priority Program ``Geometry at Infinity'', DFG 441848266, and the Research Training Group ``Algebro-Geometric Methods in Algebra, Arithmetic, and Topology'', DFG 284078965.  The second author acknowledges financial support from the N.S.F. Postdoctoral Fellowship, DMS-2103335.   We wish to thank Skip Garibaldi, Steffen Kionke, Benjamin Klopsch, and Alan Reid for helpful discussions.  Part of this work was completed while the first author visited Rice University, and he is grateful for
the department's hospitality.

\section{Galois cohomology of simple algebraic groups} \label{section:galois}

Let \(k\) be a number field and let \(\mathbf{G}\) be a simply-connected absolutely almost simple \(k\)-split linear algebraic group.  Choose a maximal \(k\)-split torus \(\mathbf{S} \subset \mathbf{G}\) and pick a set of simple roots \(\Delta \subset \Phi(\mathbf{G}, \mathbf{S})\) in the corresponding root system of \(\mathbf{G}\).  The set \(\Delta\) determines a unique Weyl chamber and hence a unique \(k\)-defined Borel subgroup \(\mathbf{B} \subset \mathbf{G}\) containing \(\mathbf{S}\)~\cite{Borel:linear-algebraic}*{Proposition~13.10\,(2), p.\,170}.  The Dynkin diagram symmetries define a subgroup \(\operatorname{Sym} \Delta\) of the permutation group of the set \(\Delta\).  Let \(\operatorname{Ad} \mathbf{G} = \mathbf{G} / Z(\mathbf{G})\) be the group of inner automorphisms of \(\mathbf{G}\).  Then we have a split short exact sequence
\begin{equation} \label{eq:inn-out} 1 \rightarrow \operatorname{Ad} \mathbf{G} \rightarrow \operatorname{Aut} \mathbf{G} \rightarrow \operatorname{Sym} \Delta \rightarrow 1. \end{equation}
There is a canonical splitting that lifts \(\operatorname{Sym} \Delta\) to the subgroup of those \(k\)-automorphisms of \(\mathbf{G}\) that fix \(\mathbf{S}\) and the Borel subgroup \(\mathbf{B} \subset \mathbf{G}\), cf.\,\cite{Platonov-Rapinchuk:algebraic-groups}*{p.\,77}.  Fix an algebraic closure \(\overline{k}/k\).  The exact sequence and the splitting evaluated at the field extension \(\overline{k}/k\) are equivariant with respect to the action of the profinite group \(\operatorname{Gal}(k) = \operatorname{Gal}(\overline{k}/k)\) where \(\operatorname{Gal}(k)\) acts trivially on \(\operatorname{Sym} \Delta\).  From this, we obtain a commutative diagram of pointed Galois cohomology sets with split exact rows so that \(i\) and \(I\) have trivial kernel while \(j\) and \(J\) are surjective~\cite{Serre:galois-cohomology}*{Prop.~38, p.\,51}.  We refer to it as the \emph{first fundamental diagram}:
\[
\begin{tikzcd}[row sep=7mm, column sep=5mm]
\textstyle\prod \limits_{v \nmid \infty} H^1(k_v, \operatorname{Ad} \mathbf{G}) \arrow{r}{I} & \textstyle\prod\limits_{v \nmid \infty} H^1(k_v, \operatorname{Aut} \mathbf{G}) \arrow{r}{J} & \textstyle\prod\limits_{v \nmid \infty} H^1(k_v, \operatorname{Sym} \Delta) \arrow[bend left=10]{l}{S} \\
{H^1(k, \operatorname{Ad} \mathbf{G})} \arrow{r}{i} \arrow{u}{f}                    & {H^1(k, \operatorname{Aut} \mathbf{G})} \arrow{r}{j} \arrow{u}{g}                    & {H^1(k, \operatorname{Sym} \Delta)}. \arrow{u}{h} \arrow[bend left=10]{l}{s} 
\end{tikzcd}
\]
To set up the diagram, we used that for group valued functor \(A\) defined on the category of algebraic field extensions of \(k\), Galois cohomology \(H^1(k, A)\) is functorial in \(k\), provided \(A\) satisfies certain compatibilities~\cite{Serre:galois-cohomology}*{Section~II.1.1}.  Note that this is a little subtle because there are no functorial models for the algebraic closure \(\overline{k}/k\).  However, different isomorphisms \(\overline{k} \cong \overline{k}'\) induce conjugate isomorphisms \(\operatorname{Gal}(\overline{k}/k) \cong \operatorname{Gal}(\overline{k}'/k)\).  Therefore, we obtain a unique isomorphism \(H^1(\overline{k}/k, A(\overline{k}/k)) \cong H^1(\overline{k}'/k, A(\overline{k}'/k))\) of pointed sets.

Now consider the short exact sequence
\begin{equation} 1 \longrightarrow \operatorname{Z}(\mathbf{G}) \longrightarrow \mathbf{G} \longrightarrow \operatorname{Ad} \mathbf{G} \longrightarrow 1 \end{equation}
which describes a nontrivial central extension of \(\operatorname{Gal}(k)\)-groups.  It induces a long exact sequence in Galois cohomology
  \begin{align} \label{eq:les}
    \begin{split}
      1  \xrightarrow{\phantom{\delta^0}} Z(\mathbf{G})(k) & \xrightarrow{\iota^0} \mathbf{G}(k) \xrightarrow{\pi^0} (\operatorname{Ad} \mathbf{G})(k) \xrightarrow{\delta^0} H^1(k, Z(\mathbf{G})) \xrightarrow{\iota^1} \\
      & \xrightarrow{\iota^1} H^1(k, \mathbf{G}) \xrightarrow{\pi^1} H^1(k, \operatorname{Ad} \mathbf{G}) \xrightarrow{\delta^1} H^2(k, Z(\mathbf{G}))
    \end{split}        
  \end{align}
  as in \cite{Serre:galois-cohomology}*{Proposition~43, p.\,55}.  From naturality of this sequence, we obtain a \emph{second fundamental diagram} in which the map \(f\) occurs:
  \[
\begin{tikzcd}
\bigoplus\limits_{v \nmid \infty} H^1(k_v, \mathbf{G}) \arrow[r, "\oplus_v\, \pi^1_v"] & \bigoplus\limits_{v \nmid \infty} H^1(k_v, \operatorname{Ad} \mathbf{G}) \arrow[r, "\oplus_v\, \delta^1_v"] & \bigoplus\limits_{v \nmid \infty} H^2(k_v, \operatorname{Z} (\mathbf{G})) \\
H^1(k, \mathbf{G}) \arrow[u] \arrow[r, "\pi^1"] & H^1(k, \operatorname{Ad} \mathbf{G}) \arrow[u, "f"] \arrow[r, "\delta^1"] & H^2(k, \operatorname{Z}(\mathbf{G})). \arrow[u, "b"]
\end{tikzcd}
\]
A priori, the rows are only exact at the middle term.  Note that in contrast to the first fundamental diagram, we use the symbol \(\bigoplus\) instead of \(\prod\) because all vertical maps have image in the subsets of the products consisting of elements with only finitely many nontrivial coordinates.  We point out that this interpretation of ``\(\bigoplus\)'' is neither the product nor the coproduct in the category of pointed sets.

\section{Proof of Theorem~\ref{thm:finite-splitting-principle}} \label{section:finite-splitting-principle}

We begin the proof of Theorem~\ref{thm:finite-splitting-principle} with some propositions as preparation.  Recall the notation from the first fundamental diagram.

\begin{proposition} \label{prop:h-trivial-kernel}
  The maps \(h\) and \(g_{|\operatorname{im} s}\) have trivial kernel.
\end{proposition}

\begin{proof}
  Since~\(S\) and~\(s\) are induced by the same section of~\eqref{eq:inn-out}, we have the commutativity \(g \circ s = S \circ h\).  Hence if we show that \(h\) has trivial kernel, so will~\(g\) when restricted to the image of~\(s\).

  If \(\mathbf{G}\) is not of type \(D_4\),  the triviality of \(\operatorname{ker} h\) follows from the more general statement that for each finite abelian group \(B\) with trivial \(\operatorname{Gal}(k)\)-action and each finite set \(S\) of places of \(k\), the map
  \[ H^1(k, B) \longrightarrow \prod_{v \notin S} H^1(k_v, B) \]
  is injective~\cite{Sansuc:groupe-de-brauer}*{Lemme\,1.1.(i)} as an easy application of Chebotarev's density theorem.  If \(\mathbf{G}\) does have type \(D_4\), then \(\operatorname{Sym} \Delta \cong S_3\) is metabelian and the short exact sequence
  \[ 1 \rightarrow \Z/3 \rightarrow S_3 \rightarrow \Z/2 \rightarrow 1 \]
  is a split extension of groups with trivial Galois action.  Hence we obtain a diagram of Galois cohomology sets
\[
\begin{tikzcd}[row sep=7mm, column sep=5mm]
\textstyle\prod \limits_{v \nmid \infty} H^1(k_v, \Z/3) \arrow[hookrightarrow]{r} & \textstyle\prod\limits_{v \nmid \infty} H^1(k_v, S_3) \arrow{r} & \textstyle\prod\limits_{v \nmid \infty} H^1(k_v, \Z/2) \\
{H^1(k, \Z/3)} \arrow[hookrightarrow]{r} \arrow[hookrightarrow]{u}                    & {H^1(k, S_3)} \arrow{r} \arrow{u}                    & {H^1(k, \Z/2)}. \arrow[hookrightarrow]{u}
\end{tikzcd}
\]
in which the ``hook arrows'' have trivial kernel.  By a straightforward diagram chase, the middle vertical arrow has trivial kernel, too.
\end{proof}

From the discussion in Section~\ref{section:galois}, the classes in the image of the sections~\(s\) and~\(S\) correspond to those forms which have Borel subgroups defined over~\(k\) and~\(k_v\), respectively, in other words to the quasi split forms.  So the proposition already shows that a finite splitting principle holds among quasi split forms of all types.  More generally, the proof showed that \(h\) and hence~\(g_{|\operatorname{im} s}\) still have trivial kernel if any finite number of places are omitted.

\begin{proposition} \label{prop:ker-equals-ker}
\(\ker g = i(\ker f)\).
\end{proposition}

\begin{proof}
Let \(\alpha \in \ker g\).  Then \(1 = J(g(\alpha)) = h(j(a))\), hence \(j(a) = 1\) because \(\operatorname{ker} h = 1\) by Proposition~\ref{prop:h-trivial-kernel}.  Exactness of the lower row in the first fundamental diagram shows that there exists \(\beta \in H^1(k, \operatorname{Ad} \mathbf{G})\) such that \(i(\beta) = \alpha\), and \(I(f(\beta)) = g(i(\beta)) = g(\alpha) = 1\).  Since \(I\) has trivial kernel, we obtain \(f(\beta) = 1\), so \(\alpha \in i(\ker f)\).  Conversely, for \(\beta \in \ker f\), we have \(g(i(\beta)) = I(f(\beta)) = I(1) = 1\).
\end{proof}

Since \(i\) has trivial kernel, we see that the finite splitting principle is equivalent to \(\ker f\) being trivial.  Therefore, we now turn our attention to the second fundamental diagram.

\begin{proposition} \label{prop:injective}
  The homomorphism of abelian groups
  \[ b \colon H^2(k, Z(\mathbf{G})) \longrightarrow \bigoplus_{v \nmid \infty} H^2(k_v, Z(\mathbf{G})) \]
  is injective if and only if one or both of the following is true:
  \begin{enumerate}[label=(\roman*)]
  \item \label{item:odd-center} \(\mathbf{G}\) has type \(A_{2n}\), \(E_6\), \(E_8\), \(F_4\), or \(G_2\),
  \item \label{item:almost-imaginary} \(k\) has at most one real place.
  \end{enumerate}
\end{proposition}

\begin{proof}
  Recall the Albert--Brauer--Hasse--Noether theorem stating that the Brauer group of \(k\) is determined by the Brauer group of the local completions of \(k\) in the sense that we have a short exact sequence of abelian groups
\begin{equation} \label{eq:brauer} 1 \longrightarrow \operatorname{Br}(k) \xrightarrow{\ \eta \ } \bigoplus_{v \in V_k} \operatorname{Br}(k_v) \longrightarrow \Q / \Z \longrightarrow 1. \end{equation}
The map to \(\Q/\Z\) sums up Hasse invariants where for a real place \(v\), the Hasse invariant of the nontrivial element in \(\operatorname{Br}(k_v) \cong \Z / 2\) is defined to be \(1/2\).  If \(Z(\mathbf{G}) \cong \mu_n\) is cyclic of order \(n\), then by~\cite{Platonov-Rapinchuk:algebraic-groups}*{Lemma~2.6, p.\,73}, the map \(b\) can be identified with the restriction of \(\eta\) to the subgroups of elements of order dividing \(n\) followed by the projection leaving out the direct summands for infinite places
\[ b \colon \operatorname{Br}(k)_n \longrightarrow \bigoplus_{v \nmid \infty} \operatorname{Br}(k_v)_n. \]
If \(k\) is totally imaginary, the Galois cohomology \(H^2(k_v, \operatorname{Z}(\mathbf{G}))\) is trivial for each \(v \mid \infty\) so that \(b\) is injective.  If \(k\) has exactly one real place~\(v_0\), the map \(b\) is still injective because whether a central simple \(k\)-algebra splits at the real place of \(k\) is determined by the behavior at the remaining places by exactness of the sequence.  Since \(\mathbf{G}\) is \(k\)-split, it has inner type, hence by \cite{Platonov-Rapinchuk:algebraic-groups}*{table on p.\,332}, it has cyclic center unless it is of type \(D_{2k}\) in which case \(Z(\mathbf{G}) \cong \mu_2 \times \mu_2\).  But in that case, \(b\) is the product of two copies of the map that we would obtain in case \(Z(\mathbf{G}) = \mu_2\), so it is injective, too.  This shows injectivity of \(b\) if condition~\ref{item:almost-imaginary} applies.

If condition~\ref{item:odd-center} applies, we have \(Z(\mathbf{G}) = \mu_n\) with \(n\) odd.  But since \(\operatorname{Br}(\R) \cong \Z / 2\), it is clear that the kernel of \(\operatorname{Br}(k) \longrightarrow \bigoplus_{v \nmid \infty} \operatorname{Br}(k_v)\) consists of order two elements only, so \(\ker b\) is trivial, regardless of the field \(k\).

Conversely, suppose \(k\) has at least two real places \(v_0\) and \(v_1\) and that \(\mathbf{G}\) has a type different from the ones listed in~\ref{item:odd-center} and different from \(D_{2k}\).  Then by~\cite{Platonov-Rapinchuk:algebraic-groups}*{table on p.\,332}, we know that \(Z(\mathbf{G}) \cong \mu_n\) with \(n\) even.  Let \(A \in \operatorname{Br}(k)\) be the unique class which does not split precisely at \(v_0\) and \(v_1\).  Then \(A\) has order two, hence \(A \in \operatorname{Br}(k)_n\) and \(A \in \ker b\) by construction.  Similarly, if \(\mathbf{G}\) has type \(D_{2k}\), then \((A,A) \in \ker b\).  So if neither \ref{item:odd-center} nor \ref{item:almost-imaginary} applies, then \(b\) is not injective.
\end{proof}

\begin{proposition} \label{prop:ker-equals-im}
Suppose that at least one of the conditions \ref{item:odd-center} and \ref{item:almost-imaginary} in Proposition~\ref{prop:injective} applies.  Then \(\ker f = \operatorname{im} \pi^1\).
\end{proposition}

\begin{proof}
The upper left set \(\bigoplus_{v \nmid \infty} H^1(k_v, \mathbf{G})\) in the second fundamental diagram is trivial by M.\,Kneser's theorem that the first Galois cohomology of simply-connected groups vanishes over nonarchimedean local fields~\cite{Platonov-Rapinchuk:algebraic-groups}*{Theorem~6.4}.  It follows that the map \(\oplus^1_v \delta_v\) has trivial kernel.  Since also \(b\) has trivial kernel by Proposition~\ref{prop:injective}, commutativity and exactness at the middle term of the lower row imply \(\ker f = \ker \delta^1 = \operatorname{im} \pi^1\).
\end{proof}

So if condition \ref{item:odd-center} or \ref{item:almost-imaginary} in Proposition~\ref{prop:injective} applies, then the finite splitting principle is equivalent to the triviality of the image of \(\pi^1\).  It is then actually equivalent to the triviality of \(H^1(k, \mathbf{G})\) as the following proposition reveals.

\begin{proposition} \label{prop:ker-pi-trivial}
  The map \(\pi^1\) has trivial kernel.
\end{proposition}

\begin{proof}
  Recall that \(\mathbf{S} \subset \mathbf{G}\) denotes a maximal \(k\)-split torus and that~\(\mathbf{S}\) is maximal overall because \(\mathbf{G}\) is a \(k\)-split group.  Therefore \(\mathbf{S}\) is self-centralizing \cite{Springer:linear-algebraic}*{Corollary~7.6.4, p.\,130} and thus contains the center \(Z(\mathbf{G})\).  By functoriality, we conclude that the map \(\iota^1\) in the long exact sequence \eqref{eq:les} factors through \(H^1(k, \mathbf{S})\).  Since \(\mathbf{S}\) splits over \(k\), the set \(H^1(k, \mathbf{S}) \cong H^1(k, \mathbf{GL_1})^{\operatorname{rank} \mathbf{G}}\) is trivial by Hilbert's Theorem~90.  Hence~\(\iota^1\) is trivial and so is \(\ker \pi^1\) by exactness.
\end{proof}

As a side remark, the map \(\delta^0\) in \eqref{eq:les} is a homomorphism of abelian groups~\cite{Serre:galois-cohomology}*{Corollary~2, p.\,53}, hence the proposition shows that
\[ (\operatorname{Ad} \mathbf{G})(k) / \pi^0 (\mathbf{G}(k)) = \operatorname{coker} \pi^0 \cong H^1(k, Z(\mathbf{G})) \cong \prod_{i=1}^s k^* / (k^*)^{n_i} \]
where \(Z(\mathbf{G}) \cong \prod_{i=1}^s \mathbb{Z} / n_i \mathbb{Z}\) can be computed as the quotient of the weight lattice by the root lattice of \((\mathbf{G}, \mathbf{S})\).  We have now finally collected all necessary facts to prove Theorem~\ref{thm:finite-splitting-principle}.

\begin{proof}[Proof of Theorem~\ref{thm:finite-splitting-principle}]
  The celebrated Hasse principle of Kneser--Harder--Chernousov~\cite{Platonov-Rapinchuk:algebraic-groups}*{Theorem~6.6} says that
\begin{equation} \label{eq:h1g} H^1(k, \mathbf{G}) \cong \prod_{v \mid \infty} H^1(k_v, \mathbf{G}). \end{equation}
Thus for a totally imaginary number field \(k\), the Galois cohomology \(H^1(k, \mathbf{G})\) is trivial.  Additionally, both \(H^1(\R, \mathbf{SL_n})\) and \(H^1(\R, \mathbf{Sp_n})\) are well known to be trivial, see for example \cite{Adams-Taibi:Galois-cohomology}*{Table~1, p.\,1093}.  This shows that under all three condition in the theorem, \(\ker g\) is trivial by Propositions~\ref{prop:ker-equals-ker} and~\ref{prop:ker-equals-im}.

  Conversely, first assume that \(k\) has at least two real places and that \(\mathbf{G}\) is not of type \(A_{2n}\).  If \(\mathbf{G}\) has type \(A_{2n+1}\) with \(n \ge 0\), \(B_n\) with \(n \ge 2\), \(C_n\) with \(n \ge 3\), \(D_n\) with \(n \ge 4\), or \(E_7\), then by Proposition~\ref{prop:injective}, there exists a nontrivial class \(\alpha \in \ker b\).  The class \(\alpha\) has a preimage \(\beta \in H^1(k, \operatorname{Ad} \mathbf{G})\) under \(\delta^1\) because \(\delta^1\) is surjective by \cite{Platonov-Rapinchuk:algebraic-groups}*{Theorem~6.20, p.\,326}.  Then \(\oplus_v \delta^1_v (f(\beta)) = b(\delta^1(\beta)) = b(\alpha) = 1\), so \(f(\beta)=1\) because \(\oplus_v \delta^1_v\) has trivial kernel, again by Kneser's theorem~\cite{Platonov-Rapinchuk:algebraic-groups}*{Theorem~6.4}.  So \(f\) has nontrivial kernel and so has \(g\) by Proposition~\ref{prop:ker-equals-ker} because \(i\) has trivial kernel.  If on the other hand \(\mathbf{G}\) has type \(E_6\), \(E_8\), \(F_4\), or \(G_2\), then \cite{Adams-Taibi:Galois-cohomology}*{Table~3, p.\,1094} shows that \(H^1(\R, \mathbf{G})\) and hence \(H^1(k, \mathbf{G})\) is nontrivial so that by Propositions~\ref{prop:ker-equals-ker}, \ref{prop:ker-equals-im}, and~\ref{prop:ker-pi-trivial}, \(\ker g\) is nontrivial.
  
  Now assume that \(k\) has precisely one real place and \(\mathbf{G}\) has neither type \(A_n\) nor \(C_n\).  Then similarly~\cite{Adams-Taibi:Galois-cohomology}*{Table~1, p.\,1093} shows that \(H^1(\R, \mathbf{G})\) and hence \(H^1(k, \mathbf{G})\) is nontrivial and again \(\ker g\) is nontrivial by Propositions~\ref{prop:ker-equals-ker}, \ref{prop:ker-equals-im}, and~\ref{prop:ker-pi-trivial}.  This exhausts all cases and the proof is complete.
\end{proof}

\section{Proof of Theorem~\ref{thm:main-theorem}} \label{section:main-theorem}

  We now explain how Theorem~\ref{thm:main-theorem} on profinitely solitary Chevalley groups follows from Theorem~\ref{thm:finite-splitting-principle} on the finite splitting principle.  First we observe that if any one of the statements~\ref{item:totima}--\ref{item:a1} in Theorem~\ref{thm:main-theorem} is true, then \(\sum_{v \mid \infty} \operatorname{rank}_{k_v} \mathbf{G} \ge 2\).  Since \(\mathbf{G}\) is moreover \(k\)-split and in particular \(k\)-isotropic, the congruence kernel of \(\mathbf{G}\) is finite~\citelist{\cite{Raghunathan:csp-i} \cite{Raghunathan:csp-ii}}.  Now assume that one of the statements~\ref{item:totima}--\ref{item:a2n} is true.  Then by Theorem~\ref{thm:finite-splitting-principle}, the finite splitting principle holds for the \(k\)-group \(\mathbf{G}\).  Choose a \(k\)-embedding \(\mathbf{G} \subset \mathbf{GL_n}\) and set \(\Gamma = \mathbf{G} \cap \mathbf{GL_n}(\mathcal{O}_k)\) where \(\mathcal{O}_k\) denotes the ring of integers of \(k\).  We may assume that \(\widehat{\Gamma}\) has trivial congruence kernel, replacing the residually finite group \(\Gamma\) with a finite index subgroup if need be.

  Let \(\Delta\) be a finitely generated residually finite group such that \(\widehat{\Gamma}\) is commensurable with \(\widehat{\Delta}\).  Replacing \(\Gamma\) and \(\Delta\) with suitable finite index subgroups, we may assume \(\widehat{\Gamma} \cong \widehat{\Delta}\).  When \(\mathbf{G}\) is not of type \(D_4\), the outer automorphism group of \(\mathbf{G}\) has order at most two.  In this case, we can let \(\Phi \in \operatorname{Aut} \mathbf{G}\) be a representative of the nontrivial class.  Then, \(\mathbf{G} \subset \mathbf{GL_n}\) can be diagonally embedded into \(\mathbf{GL_{2n}}\) as two blocks, one of them identically, the other by means of \(\Phi\).  Replacing \(\mathbf{G}\) with the image of this embedding, every automorphism is induced by conjugation in \(\mathbf{GL_{2n}}\). When \(\mathbf{G}\) is of type \(D_4\), the outer automorphism group of \(\mathbf{G}\) is isomorphic to \(S_3\). In this case, we can similarly find an embedding of \(\mathbf{G}\) into \(\mathbf{GL_{6n}}\) so that every automorphism is induced by conjugation in \(\mathbf{GL_{6n}}\). Thus the assumptions of~\cite{Spitler:profinite}*{Theorem~7.1} are met, and we obtain an embedding \(\phi \colon \Delta \rightarrow \Lambda\) in an arithmetic group \(\Lambda \subset \mathbf{H}\) of an \(l\)-form \(\mathbf{H}\) of \(\mathbf{G}\) and an isomorphism \(\mathbb{A}^f_k \cong \mathbb{A}^f_l\) over which \(\mathbf{G}(\mathbb{A}^f_k) \cong \mathbf{H}(\mathbb{A}^f_l)\).

  We remark that the statement of~\cite{Spitler:profinite}*{Theorem~7.1} actually also assumes that \(\Gamma\) has no nontrivial homomorphisms to \(Z(\mathbf{G})\), but this assumption is not strictly necessary to reach the conclusion quoted above. One way to see this is by observing that when \(\operatorname{Hom}(\Gamma, Z(\mathbf{G}))\) is nontrivial, superrigidity implies the representations of \(\Gamma\) in \(\mathbf{G}(\overline{k})\) are essentially in bijection with the elements of \(\operatorname{Hom}(\Gamma, Z(\mathbf{G}))\). However, all of these representations become the same after projecting to \(\operatorname{Ad} \mathbf{G}(\overline{k})\). One can then apply~\cite{Spitler:profinite}*{Theorem~6.2} to both \(\mathbf{G}\) and \(\operatorname{Ad} \mathbf{G}\) to see that all the produced representations of \(\Delta\) in \(\mathbf{G}(\overline{k})\) actually have their images in arithmetic subgroups of a single algebraic group \(\mathbf{H}\). So we can choose any of these representations, and after potentially passing to a finite index subgroup of \(\Delta\), we can ensure that the homomorphism \(\phi \colon \Delta \rightarrow \Lambda\) is still injective.

  Since \(k\) is locally determined, we have an isomorphism \(\sigma \colon k \xrightarrow{\cong} l\).  We then have \(\mathbf{G}(\mathbb{A}^f_k) \cong {}^\sigma \mathbf{H}(\mathbb{A}^f_k)\), so \({}^\sigma \mathbf{H}\) is a \(k\)-form of \(\mathbf{G}\) which splits at all finite places of \(k\).  Since \(\mathbf{G}\) satisfies the finite splitting principle, it follows that \(\mathbf{G} \cong_k {}^\sigma \mathbf{H}\).  So \(\Lambda\) is commensurable with \(\Gamma\).  Moreover, \(\mathbf{H}\) has CSP because \(\mathbf{G}\) does.  By construction, the injective homomorphism \(\phi \colon \Delta \rightarrow \Lambda\) has the property that \(\widehat{\phi}(\widehat{\Delta}) \subset \mathbf{H}(\mathbb{A}^f_l)\) is the closure of \(\Lambda\) in \(\mathbf{H}(\mathbb{A}^f_l)\).  Since the congruence kernel \(C(\mathbf{H})\) is finite, we find finite index subgroups \(\Lambda_0 \le \Lambda\) and \(\Delta_0 \le \Delta\) such that \(\phi(\Delta_0) \subset \Lambda_0\), such that \(\widehat{\Lambda_0}\) intersects \(C(\mathbf{H})\) trivially, and such that \(\widehat{\phi} \colon \widehat{\Delta_0} \rightarrow \widehat{\Lambda_0}\) is an isomorphism.  So either \(\phi\) embeds \(\Delta_0\) as a proper Grothendieck subgroup of \(\Lambda_0\) or \(\phi(\Delta_0) = \Lambda_0\), hence \(\Delta\) is commensurable with \(\Gamma\).

  \medskip
For case~\ref{item:qforms}, the argument as above still produces a \(k\)-form \({}^\sigma \mathbf{H}\) of \(\mathbf{G}\) which splits at all finite places of \(k\). For each of these types the finite splitting principle fails, but using the tables from \cite{Adams-Taibi:Galois-cohomology}*{Section~10} one can calculate that for types \(G_2\), \(B_3\), and \(D_4\), the only nontrivial element of \(\operatorname{im} \pi^1\) corresponds to a \(k\)-form which is anisotropic at the infinite place. Hence by Propositions~\ref{prop:ker-equals-ker} and \ref{prop:ker-equals-im}, if \({}^\sigma \mathbf{H}\) is not \(k\)-isomorphic to \(\mathbf{G}\), the arithmetic group \(\Lambda\) is a lattice in a compact group \({}^\sigma \mathbf{H}(\R)\) and hence is finite. Since \(\phi \colon \Delta \rightarrow \Lambda\) is an injection, this would contradict that \(\widehat{\Gamma}\) is commensurable with \(\widehat{\Delta}\), so we see that in fact \({}^\sigma \mathbf{H}\) must be \(k\)-isomorphic to \(\mathbf{G}\). The rest of the argument then follows as above.

For the types \(B_4\), and \(D_5\) in case~\ref{item:qforms}, there are two nontrivial elements of \(\operatorname{im} \pi^1\). One of these corresponds to a \(k\)-form which is anisotropic at the infinite place, and hence can be ruled out by the same argument as above. The other corresponds to a \(k\)-form which has rank one at the infinite place, namely \(\operatorname{Spin}(n,1;\Q)\) where \(n=8\) in the case of type \(B_4\) and \(n=9\) in the case of type \(D_5\). So if \({}^\sigma \mathbf{H}\) is not \(k\)-isomorphic to \(\mathbf{G}\), then \(\phi \colon \Delta \rightarrow \Lambda\) embeds \(\Delta\) as a Zariski-dense subgroup of an arithmetic group \(\Lambda\) which is commensurable to \(\operatorname{SO}(n,1;\Z)\). By potentially passing to a finite index subgroup, we then get a homomorphism \(\overline{\phi}: \Delta \rightarrow \operatorname{SO}(n,1;\Z)\) whose image is Zariski-dense. This means that there is some \( h \in \overline{\phi}(\Delta)\) so that \(h\) acts loxodromically on \(\mathbb{H}^n\), and hence \(\langle h \rangle < \operatorname{SO}(n,1;\Z)\) is a geometrically finite subgroup which is isomorphic to \(\Z\).  From \cite{Bergeron-Haglund-Wise:hyperplane}*{Theorem 1.4, and the remark on p.\,447}, the group \(\operatorname{SO}(n,1;\Z)\) virtually retracts onto its geometrically finite subgroups (see also \cite{Agol-Long-Reid:bianchi} and \cite{Long-Reid:retracts}).  Hence, by restriction, there is also a virtual retraction of \(\overline{\phi}(\Delta)\) onto \(\langle h \rangle \cong \Z\).  This implies there is a finite index subgroup \(\Delta_0 < \Delta\) so that \(b_1(\Delta_0) > 0\) holds for the first Betti number.  Since \(\widehat{\Gamma}\) is commensurable with \(\widehat{\Delta}\), a finite index subgroup \(\Gamma_0 \le \Gamma\) would have \(b_1(\Gamma_0) > 0\) which is a contradiction to \(\Gamma\) having CSP or alternatively property~\((T)\).  So we conclude \({}^\sigma \mathbf{H}\) must be \(k\)-isomorphic to \(\mathbf{G}\) and the rest of the argument then follows as before.

\medskip
Finally, for case~\ref{item:a1} we can appeal to the well known fact that when \(\mathbf{G}\) has type \(A_1\), the \(k\)-forms of \(\mathbf{G}\) are in bijection with the isomorphism classes of quaternion algebras over \(k\). More precisely, \(i\) and \(\delta^1\) are bijections and we can identify \(H^2(k, \operatorname{Z}(\mathbf{G})) = \operatorname{Br}(k)_2\). Then also using Proposition~\ref{prop:ker-equals-ker}, we see there is a bijection of \(\operatorname{ker} g\) with \(\operatorname{ker} b\). Then recalling
\[ b \colon \operatorname{Br}(k)_2 \longrightarrow \bigoplus_{v \nmid \infty} \operatorname{Br}(k_v)_2 \]
and the exact sequence~\eqref{eq:brauer}, we see that when \(k\) has precisely two real places \(\operatorname{ker} b\) has one nontrivial element, corresponding to the \(k\)-form which is not split exactly at both real places, while when \(k\) has precisely three real places \(\operatorname{ker} b\) has three nontrivial elements, each corresponding to a choice of two real places which are exactly those where the \(k\)-form is not split.

So when \(k\) has signature \((2,0)\) in case~\ref{item:a1}, the argument from the previous cases produces a \(k\)-form \({}^\sigma \mathbf{H}\) of \(\mathbf{G}\) which splits at all finite places of \(k\). Then if \({}^\sigma \mathbf{H}\) was not \(k\)-isomorphic to \(\mathbf{G}\), the lattice \(\Lambda\) would lie in the compact group \(\operatorname{SU}(2) \times \operatorname{SU}(2)\).  This again leads to a contradiction, and the rest of the argument follows as in the previous case.

When \(k\) has signature \((3,0)\), if \({}^\sigma \mathbf{H}\) was not \(k\)-isomorphic to \(\mathbf{G}\), \(\Lambda\) would be a lattice in \(\operatorname{SL_2}(\R) \times \operatorname{SU}(2) \times \operatorname{SU}(2)\). This would mean that \(\phi \colon \Delta \rightarrow \Lambda\) embeds \(\Delta\) Zariski-densely as a group commensurable to a Fuchsian group.  In particular this would mean there is some finite index subgroup \(\Delta_0 < \Delta\) which is a nonabelian free group or a surface group, which contradicts that \(\widehat{\Gamma}\) is commensurable with \(\widehat{\Delta}\), again because the first Betti number is a profinite invariant.  Hence \({}^\sigma \mathbf{H}\) must be \(k\)-isomorphic to \(\mathbf{G}\) and the rest of the argument follows.

Lastly, when \(k\) has signature \((2,1)\), if \({}^\sigma \mathbf{H}\) was not \(k\) isomorphic to \(\mathbf{G}\), \(\Lambda\) would be a lattice in \(\operatorname{SL_2}(\C) \times \operatorname{SU}(2) \times \operatorname{SU}(2)\). This would mean that \(\phi \colon \Delta \rightarrow \Lambda\) embeds \(\Delta\) as a group commensurable to a non-elementary  Kleinian group. In particular this would mean \(b_1(\Delta_0) > 0\) for some finite index subgroup \(\Delta_0 < \Delta\), either by virtual fibering in the finite covolume case, or by a standard half lives half dies argument in the case of infinite covolume (for example, see the proof of Proposition 7.6 in \cite{Bridson-et-al:absolute}). This is once more a contradiction, so the proof is complete.

  \section{Proof of Theorem~\ref{thm:converse}} \label{section:converse}

  In this section, we construct various pairs of profinitely commensurable but noncommensurable Chevalley groups in order to prove Theorem~\ref{thm:converse}.  To begin, we deal with the case that \(k\) is not locally determined.  This means there exists a number field \(l\) not isomorphic to~\(k\) and an isomorphism \(\mathbb{A}^f_k \cong \mathbb{A}^f_l\).  Since \(\mathbf{G}\) does not have type \(A_1\), it is a split higher rank group, thus has CSP, and so does the unique simply-connected absolutely almost simple split linear \(l\)-group \(\mathbf{H}\) with the same Cartan--Killing type as \(\mathbf{G}\).  Therefore, also using strong approximation~\cite{Platonov-Rapinchuk:algebraic-groups}*{Theorem~7.12}, the profinite completions of any two arithmetic groups \(\Gamma \le \mathbf{G}\) and \(\Lambda \le \mathbf{H}\) are commensurable with open compact subgroups of \(\mathbf{G}(\mathbb{A}^f_k)\) and \(\mathbf{H}(\mathbb{A}^f_l)\), respectively.  But \(\mathbf{G}(\mathbb{A}^f_k) \cong \mathbf{H}(\mathbb{A}^f_l)\), so \(\Gamma\) and \(\Lambda\) are profinitely commensurable.  By superrigidity~\cite{Margulis:discrete-subgroups}*{Theorem~C, Chapter~VIII}, \(\Gamma\) and \(\Lambda\) are not commensurable as they are defined over different fields.  Hence, for the rest of the proof we assume that \(k\) is locally determined.

  \medskip
	Now assume that \(k\) has at least one real place and that \(\mathbf{G}\) is either of type \(B_n\) with \(n \geq 4\) or of type \(D_n\) with \(n \geq 5\), so that \(\mathbf{G} = \operatorname{Spin}(n+1, n; k)\) or \(\mathbf{G} = \operatorname{Spin}(n, n; k)\) respectively. On the other hand, let \(\mathbf{H} = \operatorname{Spin}(n-3,n+4; k)\) (respectively \(\mathbf{H} = \operatorname{Spin}(n-4,n+4; k)\)). When \(k = \Q\) and \(n =4\) (respectively \(n = 5\)), arithmetic subgroups of \(\mathbf{H}\) are lattices in a rank one Lie group, but in any other case they will be lattices in higher rank Lie groups and hence \(\mathbf{H}\) will have CSP. Therefore, except for the omitted cases, the arithmetic groups in \(\mathbf{G}\) are profinitely commensurable, but not commensurable, with those in~\(\mathbf{H}\). This can be seen directly because the diagonal quadratic forms \(\langle 1, 1, 1, 1 \rangle\) and \(\langle -1, -1, -1, -1 \rangle\) are isometric over \(\Q_p\) for each (finite) prime \(p\).  This observation was already used by M.\,Aka to construct profinitely isomorphic groups with and without Kazhdan's property \((T)\) in~\cite{Aka:property-t}.

        \medskip
	Next we assume that \(k\) has at least one real place and that \(\mathbf{G}\) is of type \(E_6\), \(E_7\), \(E_8\), \(F_4\), or \(G_2\).  From equation \eqref{eq:h1g} and the tables of \cite{Adams-Taibi:Galois-cohomology}*{Section~10}, the set \(H^1(k, \mathbf{G})\) is nontrivial in each of these cases, and hence \(\operatorname{im} \pi^1\) is also nontrivial by Proposition~\ref{prop:ker-pi-trivial}.  More precisely, the tables show \(\pi^1_\R \colon H^1(\R, \mathbf{G}) \rightarrow H^1(\R, \operatorname{Ad} \mathbf{G})\) is injective for each of these types, again using Proposition~\ref{prop:ker-pi-trivial} for the case of \(E_7\).  Using the Hasse principle both for simply-connected and adjoint groups (equation~\eqref{eq:h1g} and \cite{Platonov-Rapinchuk:algebraic-groups}*{Theorem~6.22, p.\,336}), we conclude that \(\pi^1\) is injective.  This shows that the elements in the image of \(\pi^1\) correspond exactly to a choice of inner real forms of \(\mathbf{G}\), coming from \(H^1(\R, \mathbf{G})\), for each real place of \(k\).  Note that by \cite{Serre:galois-cohomology}*{remark at the bottom of p.\,52}, elements in the image of \(\pi^1_\R\) will correspond to simply-connected real forms \(\mathbf{G_0}\) such that \(H^1(\R, \mathbf{G_0})\) has the same number of elements as \(H^1(\R, \mathbf{G})\).  From the tables of~\cite{Adams-Taibi:Galois-cohomology}, we can thus infer that for type \(E_6\), the other element in the image of \(\pi^1_\R\) corresponds to the quasicompact real form \(E_{6(-26)}\) of real rank two, while for type \(E_7\), the other element in the image of \(\pi^1_\R\) corresponds to the hermitian real form \(E_{7(-25)}\) of real rank three.  In types \(E_8\), \(F_4\), and \(G_2\), we have \(\mathbf{G} = \operatorname{Ad}(\mathbf{G})\) so that all real forms come from \(H^1(\R, \mathbf{G})\).
	
	Now we see that when \(k = \Q\) and \(\mathbf{G}\) is of type \(F_4\), or \(G_2\), the only nontrivial elements of \(\operatorname{im} \pi^1\) correspond to \(k\)-forms \(\mathbf{H}\) such that the arithmetic groups in \(\mathbf{H}\) are lattices in rank zero or rank one Lie groups. But in every other case, there is a choice of a nontrivial element of \(\operatorname{im} \pi^1\) so that the corresponding \(k\)-form \(\mathbf{H}\) contains arithmetic groups which are lattices in a higher rank Lie group.  Noting that \(\mathbf{H}\) is an inner twist of the split form \(\mathbf{G}\), we see from \cite{Prasad-Rapinchuk:isotropic}*{Theorem~1\,(iii) and Theorem~3} that \(\mathbf{H}\) is \(k\)-isotropic.  So in these cases, \(\mathbf{H}\) has CSP and the arithmetic groups in \(\mathbf{G}\) are profinitely commensurable, but not commensurable, with those in \(\mathbf{H}\).

        \medskip
	Finally, assume that \(k\) has at least two real places and that \(\mathbf{G}\) is of type \(A_{2n+1}\) with \(n \ge 1\) or \(C_n\) with \(n \ge 2\).  As in the proof of Proposition~\ref{prop:injective}, we can find an element in \(\ker b\) which localizes to nontrivial classes in \(H^2(\R, Z(\mathbf{G}))\) precisely at the two given real places of \(k\).  Recall that \(\delta^1\) is surjective and let \(\alpha = [a]\) be a preimage in \(H^1(k, \operatorname{Ad} \mathbf{G})\) under \(\delta^1\) of this element.  Then by commutatity in the right hand square of the second fundamental diagram, \(\alpha\) localizes to a nontrivial element in \(H^1(\R, \operatorname{Ad} \mathbf{G})\) at the two real places of \(k\).  The only inner real twist of \(\operatorname{SL}_{2n+1}(\R)\) is the group \(\operatorname{SL}_{n+1}(\mathbb{H})\), so when \(\mathbf{G}\) is of type \(A_{2n+1}\) with \(n \ge 1\), this chosen class \(\alpha\) is unique by the Hasse principle for adjoint groups.  It corresponds to a \(k\)-form \(\mathbf{H}\) so that every arithmetic subgroup of \(\mathbf{H}\) is a lattice in a higher rank Lie group.  Again by \cite{Prasad-Rapinchuk:isotropic}*{Theorems~1\,(iii) and~3}, the group \(\mathbf{H}\) is \(k\)-isotropic.  When \(\mathbf{G}\) is of type \(C_n\) with \(n \ge 2\), the chosen class \(\alpha \in H^1(k, \operatorname{Ad} \mathbf{G})\) is not unique.  In fact, by \cite{Serre:galois-cohomology}*{Section~I.5.7}, the fiber \({\delta^1}^{-1}(\delta^1(\alpha))\) of \(\alpha\) is in bijective correspondence with the image of
\[ H^1(k, {}_a \mathbf{G}) \xrightarrow{{}_a \pi^1} H^1(k, \operatorname{Ad}{}_a \mathbf{G}), \]
the ``twist'' of the map \(\pi^1\) by the cocycle \(a\).  Then by the simply-connected and adjoint Hasse principles, the fiber is also in bijective correspondence with the image of
\[ \prod_{v \mid \infty} H^1(k_v, {}_a \mathbf{G}) \xrightarrow{\oplus_v {}_a \pi_v^1} \prod_{v \mid \infty} H^1(k_v, \operatorname{Ad}{}_a \mathbf{G}).\]
By construction, the group \({}_a \mathbf{G}\) does not split at the two specified real places, hence at these real places, it is isomorphic to \(\operatorname{Sp}(p,q)\) for some \(p,q\) with \(p + q = n\). From the tables of \cite{Adams-Taibi:Galois-cohomology}, the elements in the image of \(H^1(\R, \operatorname{Sp}(p,q)) \to H^1(\R, \operatorname{Ad}\operatorname{Sp}(p,q))\) correspond exactly to all the real forms \(\operatorname{Sp}(p,q)\) with \(p+q = n\). Hence we conclude that any combination of two such real forms can be realized at the two specified real places of an element in the fiber of \(\alpha\).  Choosing any two isotropic such groups (so that \(p,q \ge 1\)), we obtain a \(k\)-form \(\mathbf{H}\), which as before has \(k\)-rank at least two. Therefore, when \(\mathbf{G}\) is of type \(A_{2n+1}\) with \(n \ge 1\) or \(C_n\) with \(n \ge 2\), the constructed group \(\mathbf{H}\) has CSP and the arithmetic groups in \(\mathbf{G}\) are profinitely commensurable, but not commensurable, with those in \(\mathbf{H}\).  This completes the proof of Theorem~\ref{thm:converse}.

 \bigskip 
 To conclude, we observe that the occurring rings in the first list of examples below Theorem~\ref{thm:converse} are all rings of integers in number fields of degree at most three.  Number fields of degree at most six are known to be \emph{arithmetically solitary}: They are determined by the Dedekind zeta function and in particular locally determined~\cite{Klingen:similarities}*{Theorem~1.16\,(d), p.\,93}.  All the underlying algebraic groups in the first list are simply-connected, absolutely almost simple \(\Q\)-split linear algebraic \(\Q\)-groups.  The group \(\mathbf{SL_n}\) has type \(A_{n-1}\), the group \(\mathbf{Spin(n+1,n)}\) has type \(B_n\), the group \(\mathbf{Sp_n}\) has type \(C_n\), and the group \(\mathbf{Spin(n,n)}\) has type \(D_n\).  With these observations, we see that the groups in the first list satisfy the assumptions of Theorem~\ref{thm:main-theorem}.

 For the groups in the second list, we additionally note that the number fields \(\Q(\sqrt[8]{7})\) and \(\Q(\sqrt[4]{2})\) are monogenic and the field generators \(\sqrt[8]{7}\) and \(\sqrt[4]{2}\) generate a power integral basis, respectively, just like in the case of \(\Q(\sqrt[3]{2})\).  This can for instance be checked by a general criterion that was recently given by H.\,Smith~\cite{Smith:monogeneity}.  Hence \(\mathbb{Z}[\sqrt[8]{7}]\) and \(\mathbb{Z}[\sqrt[4]{2}]\) are the rings of integers in \(\Q(\sqrt[8]{7})\) and \(\Q(\sqrt[4]{2})\), respectively.  The number field \(\Q(\sqrt[8]{7})\), however, is not locally determined: The ring of finite adeles is isomorphic to the one of \(\Q(\sqrt[8]{112})\) as was proven by Komatsu~\cite{Komatsu:adele-rings}.  With these remarks, we see that the groups in the second list fall under the assumptions of Theorem~\ref{thm:converse}.

\end{document}